\documentclass{amsart}
\usepackage{amssymb,amscd,amsthm}
\usepackage{latexsym}
\date{\today}

\newcommand{\Z}{{\mathbb Z}}
\newcommand{\R}{{\mathbb R}}
\newcommand{\C}{{\mathbb C}}

\newcommand{\T}{{\mathbb T}}

\newtheorem{theorem}{Theorem}

\newtheorem{prop}{Proposition}

\newtheorem{definition}{Definition}
\begin{document}
\title{Generic Continuous Spectrum for Ergodic Schr\"odinger Operators}

\author[M.\ Boshernitzan]{Michael Boshernitzan}

\address{Department of Mathematics, Rice University, Houston, TX~77005, USA}

\email{michael@rice.edu}

\author[D.\ Damanik]{David Damanik}

\address{Department of Mathematics, Rice University, Houston, TX~77005, USA}

\email{damanik@rice.edu}

\thanks{D.\ D.\ was supported in part by NSF grant
DMS--0653720.}

\begin{abstract}
We consider discrete Schr\"odinger operators on the line with
potentials generated by a minimal homeomorphism on a compact
metric space and a continuous sampling function. We introduce the
concepts of topological and metric repetition property. Assuming
that the underlying dynamical system satisfies one of these
repetition properties, we show using Gordon's Lemma that for a
generic continuous sampling function, the associated Schr\"odinger
operators have no eigenvalues in a topological or metric sense,
respectively. We present a number of applications, particularly to
shifts and skew-shifts on the torus.
\end{abstract}

\maketitle

\section{Introduction}

In this paper we study Schr\"odinger operators acting in
$\ell^2(\Z)$ by
\begin{equation}\label{oper}
(H_\omega \psi)(n) = \psi(n+1) + \psi(n-1) + V_\omega(n) \psi(n),
\end{equation}
where the potential $V_\omega$ is generated by some homeomorphism
$T$ of a compact metric space $\Omega$ and a continuous sampling
function $f : \Omega \to \R$ as follows:
\begin{equation}\label{pot}
V_\omega (n) = f(T^n \omega) , \quad \omega \in \Omega, \; n \in
\Z.
\end{equation}
We assume throughout this paper that $(\Omega,T)$ is minimal, that
is, the $T$-orbit of every $\omega \in \Omega$ is dense in $\Omega$.
A strong approximation argument shows that the spectrum of
$H_\omega$ does not depend on $\omega$. Since the case of periodic
potentials is well understood, we assume in addition that $\Omega$
is not finite. This ensures that $T$ does not have any periodic
points.

If $\mu$ is a $T$-ergodic Borel probability measure, it is known
that the spectral type of $H_\omega$ is $\mu$-almost surely
independent of $\omega$ (see, e.g., Carmona-Lacroix \cite{cl}). In
general, however, the spectral type is not globally independent of
$\omega$ (see Jitomirskaya-Simon \cite{js} for counterexamples).

Popular examples include shifts,
\begin{equation}\label{qpmodels}
\Omega = \T^d , \; T \omega = \omega + \alpha
\end{equation}
and skew-shifts
\begin{equation}\label{ssmodels}
\Omega = \T^2 , \; T(\omega_1,\omega_2) = (\omega_1 + 2\alpha ,
\omega_1 + \omega_2)
\end{equation}
on the torus. Here, we write $\T = \R / \Z$. Operators with
potentials generated by shifts on the torus are called
quasi-periodic.

In these examples, minimality of $(\Omega,T)$ holds if and only if
the coordinates of $\alpha$ together with $1$ are independent over
the rational numbers. Moreover, in these cases, normalized Lebesgue
measure on the torus in question is the unique $T$-ergodic Borel
probability measure and hence one is particularly interested in
identifying the spectrum and the spectral type of $H$ for Lebesgue
almost every $\omega$.

We refer the reader to Bourgain's recent book \cite{b5} for the
current state of the art, especially for these two classes of
models.

While the quasi-periodic case has been studied heavily for several
decades and many fundamental results have been obtained, the case
of the skew-shift is much less understood. It is of interest as it
naturally arises in the study of the quantum kicked rotor model;
compare \cite{b4}. The general expected picture is that, while the
skew-shift model seems to be formally close to a quasi-periodic
model, on the operator level one observes much different behavior
-- similar to that found for random potentials. Namely, while any
spectral type occurs naturally for quasi-periodic models and their
spectra have a tendency to be Cantor sets, it is expected that
skew-shift models ``almost always'' have pure point non-Cantor
spectrum. Thus, interest in Schr\"odinger operators generated by
the skew-shift is also triggered by the question of how much one
can reduce the randomness of the potentials until one fails to
observe complete localization phenomena. Such a transition could
occur at or near skew-shift models.

To illustrate this point, let us recall some conjectures from
\cite[p.~114]{b5}. Consider the potential
\begin{equation}\label{bourgpot}
V(n) = \lambda \cos (2\pi(\omega_1 + \omega_2 n + \alpha n(n-1))).
\end{equation}
Then, one expects that for every $\lambda \not= 0$, Diophantine
$\alpha \in \T$, and Lebesgue almost every $\omega \in \T^2$, the
operator \eqref{oper} with potential \eqref{bourgpot} has positive
Lyapunov exponents,\footnote{Lyapunov exponents measure the
averaged rate of exponential growth of the so-called transfer
matrices. Since we will not need them in our study, we omit the
exact definition.} pure point spectrum with exponentially decaying
eigenfunctions, and its spectrum has no gaps. Generally, one
expects further that these properties persist even when the cosine
is replaced by a sufficiently regular function $f$.

There are very few positive results in this direction. For
example, Bourgain, Goldstein, and Schlag proved a localization
result for analytic $f = \lambda g$ and sufficiently large
$\lambda$ and Bourgain proved the existence of some point spectrum
for \eqref{bourgpot} with small $\lambda$ and certain
$(\alpha,\omega) \in \T^3$; see \cite{b2,bgs}. Some results (that,
however, do not determine the spectral type) assuming weaker
regularity can be found in the paper \cite{cgs} by Chan,
Goldstein, and Schlag.

Naturally, negative results that point out limitations to the
scope in which the expected properties actually hold are of
interest as well. One result of this kind is obtained in the work
\cite{abd} by Avila, Bochi, and Damanik, where it is shown that
for the skew-shift model (and generalizations thereof), the
spectrum is a Cantor set for a residual set of continuous sampling
functions. Another result showing that expected phenomena may not
occur can be found in the paper \cite{b1} by Bjerkl\"ov. He showed
that even for (large) analytic sampling functions, the Lyapunov
exponent may vanish. Further negative results, concerning the
spectral type, will be established in the present paper. Namely,
we will show under reasonably weak assumptions that for a generic
continuous sampling function $f$, the point spectrum is empty.

The present paper should be regarded as a companion piece to work by
Avila and Damanik \cite{ad}. They proved that the absolutely
continuous spectrum is generically empty. Putting these two results
together, it follows that for a large class of ergodic Schr\"odinger
operators, the generic spectral type is singular continuous. This
will be discussed in more detail in Subsection~\ref{scspec}.

A bounded potential $V : \Z \to \R$ is called a Gordon potential if
there are positive integers $q_k \to \infty$ such that
$$
\max_{1 \le n \le q_k} |V(n) - V(n \pm q_k)| \le k^{-q_k}
$$
for every $k \ge 1$.\footnote{Here, $|V(n) - V(n \pm q_k)| \le
k^{-q_k}$ is shorthand for having both $|V(n) - V(n + q_k)| \le
k^{-q_k}$ and $|V(n) - V(n - q_k)| \le k^{-q_k}$.} Equivalently,
there are positive integers $q_k \to \infty$ such that
$$
\forall \, C > 0 : \lim_{k \to \infty} \max_{1 \le n \le q_k}
|V(n) - V(n \pm q_k)| C^{q_k} = 0.
$$
This condition is of interest because it ensures the absence of
point spectrum. That is, if $V$ is a Gordon potential, then the
Schr\"odinger operator in $\ell^2(\Z)$ with potential $V$ has no
eigenvalues. In fact, it is even true that for every $E \in \C$,
the difference equation
$$
u(n+1) + u(n-1) + V(n) = E u(n)
$$
has no non-trivial solution $u$ with $\lim_{|n| \to \infty} u (n)
= 0$; compare \cite{cfks,dp,g}.

For potentials generated by a shift on the circle, $T : \T \to
\T$, $\omega \mapsto \omega + \alpha$, it is not hard to show that
for $f$ having a fixed modulus of continuity and $\alpha$ from an
explicit residual subset of $\T$ (which depends only on the
modulus of continuity and has zero Lebesgue measure), $V_\omega$
is a Gordon potential for every $\omega$. For the specific case of
the cosine, this observation was used by Avron and Simon \cite{as}
to exhibit explicit quasi-periodic operators with purely singular
continuous spectrum; see also \cite{cfks}.

Our goal is to exhibit a variety of situations, in the general
context we consider, where the potentials satisfy the Gordon
condition, either on a residual subset of $\Omega$ or on a subset
of $\Omega$ that has full $\mu$ measure.

\begin{definition}
A sequence $\{\omega_k\}_{k \ge 0}$ in the compact metric space
$\Omega$ has the repetition property if for every $\varepsilon >
0$ and $r > 0$, there exists $q \in \Z_+$ such that
$\mathrm{dist}(\omega_k , \omega_{k+q}) < \varepsilon$ for $k = 0,
1, 2 , \ldots , \lfloor rq \rfloor$.
\end{definition}

Of course, this definition makes sense in any metric space; but we
remark that the compactness of $\Omega$ implies that the validity
of the repetition property for any given sequence in $\Omega$ is
independent of the choice of the metric.

\begin{definition}
We denote the set of points in the compact metric space $\Omega$
whose forward orbit with respect to the minimal homeomorphism $T$
has the repetition property by $PRP(\Omega,T)$. That is,
$$
PRP(\Omega,T) = \{ \omega \in \Omega : \{ T^k \omega \}_{k \ge 0}
\text{ has the repetition property} \}.
$$
We say that $(\Omega,T)$ satisfies the topological repetition
property {\rm (TRP)} if $PRP(\Omega,T) \not= \emptyset$.
\end{definition}

It is not hard to see that since $(\Omega,T)$ is minimal,
$PRP(\Omega,T)$ being non-empty actually implies that it is
residual. One can of course define $PRP(\Omega,T)$ also for a
non-minimal topological dynamical system $(\Omega,T)$. In general,
$PRP(\Omega,T)$ is always a $G_\delta$ set and so, in particular,
a dense $G_\delta$ set in the closure of the orbit of any of its
elements.

\begin{theorem}\label{rpthm}
Suppose $(\Omega,T)$ satisfies {\rm (TRP)}. Then there exists a
residual subset $\mathcal{F}$ of $C(\Omega)$ such that for every
$f \in \mathcal{F}$, there is a residual subset $\Omega_f
\subseteq \Omega$ with the property that for every $\omega \in
\Omega_f$, $V_\omega$ defined by \eqref{pot} is a Gordon
potential.
\end{theorem}

Combining this theorem with Gordon's result, we find that
$H_\omega$ has purely continuous spectrum for topologically
generic $f \in C(\Omega)$ and $\omega \in \Omega$ whenever
$(\Omega,T)$ satisfies (TRP).

Let us now fix some $T$-ergodic measure $\mu$. As mentioned above,
given the general theory of ergodic Schr\"odinger operators
(cf.~\cite{cl}), it is a natural goal to identify the almost sure
spectral type with respect to $\mu$. By almost sure independence, it
is sufficient to consider sets of positive $\mu$ measure. Thus, we
are seeking a criterion that implies the applicability of the
Gordon's lemma at least on a set of positive $\mu$ measure.

\begin{definition}
We say that $(\Omega,T,\mu)$ satisfies the metric repetition
property {\rm (MRP)} if $\mu(PRP(\Omega,T)) > 0$.
\end{definition}

The following result shows that (MRP) implies the desired Gordon
property on a full measure set.

\begin{theorem}\label{srpthm}
Suppose $(\Omega,T,\mu)$ satisfies {\rm (MRP)}. Then there exists a
residual subset $\mathcal{F}$ of $C(\Omega)$ such that for every $f
\in \mathcal{F}$, there is a subset $\Omega_f \subseteq \Omega$ of
full $\mu$ measure with the property that for every $\omega \in
\Omega_f$, $V_\omega$ defined by \eqref{pot} is a Gordon potential.
\end{theorem}

Naturally, the strongest repetition property, defined in terms of
the size of the set $PRP(\Omega,T)$, that can hold is the
following:

\begin{definition}
We say that $(\Omega,T)$ satisfies the global repetition property
{\rm (GRP)} if $PRP(\Omega,T) = \Omega$.
\end{definition}

It is an interesting question whether (GRP) implies a stronger
statement for the associated potentials. Namely, when (GRP) holds,
is it true that for generic continuous $f$, $V_\omega$ is a Gordon
potential for every $\omega \in \Omega$? We believe the answer is
no but have been unable to prove this.

Let us explore when (TRP), (MRP), and (GRP) hold for our two
classes of examples. It will turn out that, in these examples,
(TRP), (MRP), and (GRP) always hold or fail simultaneously.
Trivially, we have that (GRP) $\Rightarrow$ (MRP) $\Rightarrow$
(TRP), but one can ask whether some reverse implication holds in
general. This turns out not to be the case, see the discussion in
Subsection~\ref{ss.remarks}.

For minimal shifts of the form \eqref{qpmodels}, the situation is
particularly nice as the following result shows.

\begin{theorem}\label{shiftrpthm}
Every minimal shift $T \omega = \omega + \alpha$ on the torus
$\T^d$ satisfies {\rm (GRP)}, and hence also {\rm (TRP)} and {\rm
(MRP)}.
\end{theorem}

As a consequence, we find that for every minimal shift on $\T^d$ and
a generic function $f \in C(\T^d)$, the operator $H_\omega$ has
empty point spectrum for almost every $\omega$. This is especially
surprising if a coupling constant is introduced. Generally, one
expects pure point spectrum at large coupling, but our proof
excludes point spectrum for all values of the coupling constant at
once! This is discussed in more detail in
Subsection~\ref{couplingconstant}. We also want to mention that
recent results indicate that point spectrum should become more and
more prevalent as the dimension of the torus increases \cite{b3}.
Our result on the absence of point spectrum, on the other hand,
holds for all torus dimensions.

Let us now turn to minimal skew-shifts of the form
\eqref{ssmodels}. Recall that $\alpha \in \T$ is called badly
approximable if there is a constant $c > 0$ such that
$$
\langle \alpha q \rangle > \frac{c}{q}
$$
for every $q \in \Z \setminus \{0\}$. Here, we write $\langle x
\rangle = \mathrm{dist}_\T(x,0)$ ($= \min \{ |x - p| : p \in \Z
\}$, where $x$ denotes any representative in $\R$). The set of
badly approximable $\alpha$'s has zero Lebesgue measure; see, for
example, \cite[Theorem~29 on p.~60]{khin}. In terms of the
continued fraction expansion of $\alpha$ (cf.~\cite{khin}), being
badly approximable is equivalent to having bounded partial
quotients.

\begin{theorem}\label{sshiftrpthm}
For a minimal skew-shift $T(\omega_1,\omega_2) = (\omega_1 +
2\alpha , \omega_1 + \omega_2)$ on the torus $\T^2$, the following
are equivalent:
\begin{itemize}

\item[{\rm (i)}] $\alpha$ is not badly approximable.

\item[{\rm (ii)}] $(\Omega,T)$ satisfies {\rm (GRP)}.

\item[{\rm (iii)}] $(\Omega,T,\mathrm{Leb})$ satisfies {\rm
(MRP)}.

\item[{\rm (iv)}] $(\Omega,T)$ satisfies {\rm (TRP)}.

\end{itemize}
\end{theorem}

Thus, for Lebesgue almost every $\alpha$, the operator $H_\omega$
generated by the corresponding skew-shift and a generic function
$f \in C(\T^2)$ has empty point spectrum for Lebesgue almost every
$\omega$. This is surprising given that the expected spectral type
for operators generated by the skew shift is pure point. Again,
one can introduce a coupling constant and absence of point
spectrum then holds for all values of the coupling constant
simultaneously. Let us also emphasize that, to the best of our
knowledge, our result provides the first examples of Schr\"odinger
operators with potentials defined by a skew-shift that have empty
point spectrum. In particular, the expected localization result
for such operators will need a suitable regularity assumption for
the sampling function.

The paper is organized as follows. In Section~\ref{secrpgp}, we
establish the relation between the topological or metric
repetition property for the underlying dynamical system and the
Gordon property for the associated potentials when a generic
continuous sampling function is chosen; that is, we prove
Theorems~\ref{rpthm} and \ref{srpthm}. The validity of the
topological, metric, or global repetition property for the two
classes of examples is then explored in Section~\ref{secsss},
where we prove Theorems~\ref{shiftrpthm} and \ref{sshiftrpthm}. We
conclude the paper with some further results and comments in
Section~\ref{secfrc}.

\section{Repetition Properties and Gordon
Potentials}\label{secrpgp}

In this section we prove Theorems~\ref{rpthm} and \ref{srpthm}.

\begin{proof}[Proof of Theorem~\ref{rpthm}.]
By assumption, there is a point $\omega \in \Omega$ whose forward
orbit has the repetition property. For each $k \in \Z_+$, consider
$\varepsilon = \frac{1}{k}$, $r = 3$, and the associated $q_k =
q(\varepsilon,r)$. This ensures
$$
q_k \to \infty.
$$
Take an open ball $B_k$ around $\omega$ with radius small enough so
that
$$
\overline{T^n(B_k)} , \; 1 \le n \le 4q_k
$$
are disjoint and, for every $1 \le j \le q_k$,
$$
\bigcup_{l = 0}^3 T^{j+lq_k} (B_k)
$$
is contained in some ball of radius $4\varepsilon$. Define
$$
\mathcal{C}_k = \left\{ f \in C(\Omega) : f \text{ is constant on
each set } \bigcup_{l = 0}^3 T^{j+lq} (B_k) , \; 1 \le j \le q_k
\right\}
$$
and let $\mathcal{F}_k$ be the open $k^{-q_k}$ neighborhood of
$\mathcal{C}_k$ in $C(\Omega)$. Notice that for each $m$,
$$
\bigcup_{k \ge m} \mathcal{F}_k
$$
is an open and dense subset of $C(\Omega)$. This follows since every
$f \in C(\Omega)$ is uniformly continuous and the diameter of the
set $\bigcup_{l = 0}^3 T^{j+lq_k} (B_k)$ goes to zero, uniformly in
$j$, as $k \to \infty$. Thus,
$$
\mathcal{F} = \bigcap_{m \ge 1} \bigcup_{k \ge m} \mathcal{F}_k
$$
is a dense $G_\delta$ subset of $C(\Omega)$.

Consider some $f \in \mathcal{F}$. Then, $f \in \mathcal{F}_{k_l}$
for some sequence $k_l \to \infty$. Observe that for every $m \ge
1$,
$$
\bigcup_{l \ge m} \bigcup_{j = 1}^{q_{k_l}} T^{j + q_{k_l}}(B_{k_l})
$$
is an open and dense subset of $\Omega$ since $T$ is minimal and
$q_{k_l} \to \infty$. Thus,
$$
\Omega_f = \bigcap_{m \ge 1} \bigcup_{l \ge m} \bigcup_{j =
1}^{q_{k_l}} T^{j + q_{k_l}}(B_{k_l})
$$
is a dense $G_\delta$ subset of $\Omega$.

It now readily follows that for every $f \in \mathcal{F}$ and
$\omega \in \Omega_f$, $V_\omega$ is a Gordon potential. Explicitly,
since $\omega \in \Omega_f$, $\omega$ belongs to $\bigcup_{j =
1}^{q_{k_l}} T^{j + q_{k_l}}(B_{k_l})$ for infinitely many $l$. For
each such $l$, we have by construction that
$$
\max_{1 \le j \le q_{k_l}} |f(T^{j} \omega) - f(T^{j + q_{k_l}}
\omega)| < 2 k_l^{-q_{k_l}}
$$
and
$$
\max_{1 \le j \le q_{k_l}} |f(T^{j} \omega) - f(T^{j - q_{k_l}}
\omega)| < 2 k_l^{-q_{k_l}}
$$
This shows that $V_\omega(n) = f(T^n \omega)$ is a Gordon potential.
\end{proof}

\begin{proof}[Proof of Theorem~\ref{srpthm}.]
Notice that, by ergodicity, the assumption $\mu(PRP(\Omega,T)) > 0$
implies $\mu(PRP(\Omega,T)) = 1$.

Let us inductively define a sequence of positive integers,
$\{n_i\}_{i \ge 1}$, and a sequence of subsets of $\Omega$, $\{
\Omega_i \}_{i \ge 1}$.

Since $\mu(PRP(\Omega,T)) = 1$ we can choose $n_1 \in \Z_+$ large
enough so that\footnote{The definition of $\Omega_1$ contains
redundancies whose purpose is to motivate the definition of the
subsequent sets $\Omega_i$.}
\begin{align*}
\Omega_1 = \Big\{ \omega \in \Omega : & \text{ there exists } n
\in [1,n_1) \text{ such that for } k_1 , k_2 \in [1, n] \\
& \text{ with } k_1 \equiv k_2 \!\!\!\!\! \mod n, \text{ we have }
\mathrm{dist}(T^{k_1} \omega , T^{k_2} \omega) < 1  \Big\}
\end{align*}
obeys
$$
\mu (\Omega_1) > 1 - 2^{-1}.
$$
Once $n_{i-1}$ and $\Omega_{i-1}$ have been determined, we can
define $n_i$ and $\Omega_i$ as follows. It is possible to find $n_i
> \max \{ n_{i-1} , 2^i \}$ such that
\begin{align*}
\Omega_i = \Big\{ \omega \in \Omega : & \text{ there exists } n
\in [n_{i-1},
n_i) \text{ such that for } k_1 , k_2 \in [1, i n] \\
&  \text{ with } k_1 \equiv k_2 \!\!\!\!\! \mod n, \text{ we have
} \mathrm{dist}(T^{k_1} \omega , T^{k_2} \omega) < 2^{-i} \Big\}
\end{align*}
obeys
\begin{equation}\label{mesest1}
\mu(\Omega_i) > 1 - 2^{-i}.
\end{equation}

For $i \ge 5$, we set $m_i = (i n_i)^2$ and choose (using the
Rokhlin-Halmos Lemma \cite[Theorem~1 on p.~242]{cfs}) $O_i \subset
\Omega$ in a way that $T^j O_i$, $1 \le j \le m_i$ are disjoint and
$$
\mu \left( \bigcup_{j = 1}^{m_i} T^j O_i \right)
> 1 - 2^{-n_i - 1}.
$$
Next we further partition $O_i$ into sets $S_{i,l}$, $1 \le l \le
s_i$ such that for every $0 \le u \le m_i$,
\begin{equation}\label{diamest}
\mathrm{diam}(T^u S_{i,l}) < \frac{1}{n_i}.
\end{equation}
Choose $K_{i,l} \subseteq S_{i,l}$ compact with
$$
\mu(K_{i,l}) > \mu(S_{i,l}) \left( 1 - 2^{-n_i - 1} \right).
$$
Then, $T^u K_{i,l}$,  $0 \le u \le m_i$, $1 \le l \le s_i$ are
disjoint and their total measure is
\begin{equation}\label{mesest2}
\mu \left( \bigcup_{u,l}  T^u K_{i,l} \right) \ge (1 - 2^{-n_i -
1})^2 > 1 - 2^{-n_i}.
\end{equation}

We will now collect a large subfamily of $\{T^u K_{i,l}\}$. For
each $l$, we let $u$ run from $0$ upwards and ask for the
corresponding $T^u K_{i,l}$ whether it has non-empty intersection
with $\Omega_i$. If it does, then there is a point $\omega$ and a
corresponding $n \in (n_{i-1},n_i)$. We add the current $T^u
K_{i,l}$ to the subfamily we construct, along with the sets
$T^{u+h} K_{i,l}$, $1 \le h \le in$. Then we continue with
$T^{u+in+1} K_{i,l}$ and do the same. We stop when we are within
$n_i$ steps of the top of the tower. Let us denote the subfamily
so constructed by $\mathcal{K}_i$. By \eqref{mesest1},
\eqref{mesest2}, and $T$-invariance of $\mu$, we have that
\begin{equation}\label{mesest3}
\mu \left( \bigcup_{T^u K_{i,l} \in \mathcal{K}_i} T^u K_{i,l}
\right) > 1 - 2^{-n_i} - 2^{-i} - \frac{n_i}{m_i}.
\end{equation}

The next step is to group each of these ``runs'' into arithmetic
progressions. Notice that locally we have $n$ consecutive points
that are $(i-1)$ times repeated up to some small error. Notice that
this extends to the entire set if we make the allowed error a bit
larger. Let us group these $in$ sets into $n$ arithmetic
progressions of length $i$. The union of each of these arithmetic
progressions of sets will constitute a new set $C_{i,m}$. By the
definition of $\Omega_i$ and \eqref{diamest}, we have
\begin{equation}\label{diamest2}
\mathrm{diam}(C_{i,m}) < \frac{2}{n_i} + 2^{-i}.
\end{equation}
We will also consider the sets $\tilde C_{i,m}$ that are defined
similarly, but with the first and the last set in the corresponding
sequence of $i$ sets deleted.

We can now continue as before. Define
$$
F_i = \{ f \in C(\Omega) : f \text{ is constant on each set }
C_{i,m} \}
$$
and let $\mathcal{F}_i$ be the $i^{-n_i}$ neighborhood of $F_i$ in
$C(\Omega)$. Notice that for each $m$,
$$
\bigcup_{i \ge m} \mathcal{F}_i
$$
is an open and dense subset of $C(\Omega)$. This follows from
\eqref{mesest3} and \eqref{diamest2} since every $f \in C(\Omega)$
is uniformly continuous. Thus,
$$
\mathcal{F} = \bigcap_{m \ge 1} \bigcup_{i \ge m} \mathcal{F}_i
$$
is a dense $G_\delta$ subset of $C(\Omega)$.

If $f \in \mathcal{F}$, there is a sequence $i_k \to \infty$ such
that $f \in \mathcal{F}_{i_k}$. For each $k$, $f$ is within
$i_k^{-n_{i_k}}$ of being constant on each set $C_{i_k,m}$. Recall
that this set is the union of $i$ sets in arithmetic progression
relative to $T$.

If we instead consider $\tilde C_{i_k,m}$, we can go forward and
backward one period and hence, by construction, this is exactly the
Gordon condition at this level. Thus, it only remains to show that
almost every $\omega \in \Omega$ belongs to infinitely many $\tilde
C_{i_k,m}$. This, however, follows from the measure estimates
obtained above and the Borel-Cantelli Lemma.
\end{proof}

\section{Repetition Properties for Shifts and
Skew-Shifts}\label{secsss}

In this section we consider minimal shifts and skew-shifts and
identify those cases that obey (TRP), (MRP), and (GRP). That is,
we prove Theorems~\ref{shiftrpthm} and \ref{sshiftrpthm}.

\begin{proof}[Proof of Theorem~\ref{shiftrpthm}.]
By assumption, the orbit of $0 \in \T^d$ is dense. In particular,
we can define $q_k \to \infty$ such that $T^{q_k}(0)$ is closer to
$0$ than any point $T^n(0)$, $1 \le n < q_k$. In particular, for
every $\varepsilon > 0$ and every $r > 0$, there is
$k(\varepsilon,r)$ such that for $k \ge k(\varepsilon,r)$,
$T^{q_k}$ is a shift on $\T^d$ with a shift vector of length
bounded by $\varepsilon$. The repetition property now follows for
the forward orbit of any choice of $\omega \in \T^d$. Thus, (GRP)
is satisfied.
\end{proof}

\noindent\textit{Remark.} The proof only used that the shift on
the torus is an isometry. Thus, the result extends immediately to
any minimal isometry of a compact metric space.

\begin{proof}[Proof of Theorem~\ref{sshiftrpthm}.]
Iterating the skew-shift $n$ times, we find
$$
T^n (\omega_1,\omega_2) = ( \omega_1 + 2 n \alpha , \omega_2 + 2 n
\omega_1 + n(n-1)\alpha ).
$$
Thus,
\begin{equation}\label{skewshiftdiff}
T^{n+q} (\omega_1 ,\omega_2) - T^n (\omega_1,\omega_2) = (
2q\alpha , 2q \omega_1 + q^2 \alpha + 2nq \alpha - q \alpha ).
\end{equation}

\noindent\underline{(i) $\Rightarrow$ (ii)}: Assume that $\alpha$
is not badly approximable. This means that there is some sequence
$q_k \to \infty$, such that
\begin{equation}\label{qkchoice}
\lim_{k \to \infty} q_k \langle \alpha q_k \rangle = 0.
\end{equation}

Let $(\omega_1,\omega_2) \in \T^2$, $\varepsilon > 0$, and $r > 0$
be given. We will construct a sequence $\tilde q_k \to \infty$ so
that for $1 \le n \le r \tilde q_k$,
\begin{equation}\label{diffexp}
( 2 \tilde q_k \alpha , 2 \tilde q_k \omega_1 + \tilde q_k^2
\alpha + 2 n \tilde q_k \alpha - \tilde q_k \alpha )
\end{equation}
is of size $O(\varepsilon)$. Each $\tilde q_k$ will be of the form
$m_k q_k$ for some $m_k \in \{ 1 , 2 , \ldots , \lfloor
\varepsilon^{-1} \rfloor + 1 \}$.

It follows from \eqref{qkchoice} that in \eqref{diffexp}, every
term except $2 \tilde q_k \omega_1$ goes to zero as $k \to
\infty$, regardless of the choice of $m_k$, and hence is less than
$\varepsilon$ for $k$ large enough. To treat the remaining term,
we can just choose $m_k$ in the specified $\varepsilon$-dependent
range so that $2 \tilde q_k \omega_1 = m_k (2 q_k \omega_1)$ is of
size less than $\varepsilon$ as well. Consequently, by
\eqref{skewshiftdiff}, the orbit of $(\omega_1,\omega_2)$ has the
repetition property. Since $(\omega_1,\omega_2)$ was arbitrary, it
follows that (GRP) holds.
\\[3mm]
\underline{(ii) $\Rightarrow$ (iii)}: This is immediate.
\\[3mm]
\underline{(iii) $\Rightarrow$ (iv)}: This is immediate.
\\[3mm]
\underline{(iv) $\Rightarrow$ (i)}: Assuming (TRP), we see that
there is a point $\omega$ such that $\{ T^n \omega \}_{n \ge 0}$
has the repetition property. In particular, by
\eqref{skewshiftdiff}, we see that for every $\varepsilon > 0$,
there are $q_k \to \infty$ so that
\begin{equation}\label{pf3.1}
\langle 2q_1 \omega_1 + q_k^2 \alpha + 2nq_k \alpha - q_k \alpha
\rangle < \varepsilon \text{ for } 0 \le n \le q_k.
\end{equation}
Evaluating this for $n = 0$, we find that $\langle 2q_1 \omega_1 +
q_k^2 \alpha  - q_k \alpha \rangle < \varepsilon$. Now vary $n$.
Each time we increase $n$, we shift in the same direction by
$\langle 2 q_k \alpha \rangle$. If $\varepsilon > 0$ is
sufficiently small, it follows from the estimate \eqref{pf3.1}
that we cannot go around the circle completely and hence we have
$\langle 2nq_k \alpha \rangle = n \langle 2q_k \alpha \rangle$ for
every $0 \le n \le q_k$. We find that $\langle \alpha q_k \rangle
\lesssim \frac{\varepsilon}{q_k}$, which shows that $\alpha$ is
not badly approximable.
\end{proof}

\section{Further Results and Comments}\label{secfrc}

\subsection{Interval Exchange Transformations}

Of course, it is interesting to explore the validity of the
various repetition properties for other underlying dynamical
systems. In this subsection we will briefly discuss interval
exchange transformations as these dynamical systems have been
studied in the context of ergodic Schr\"odinger operators before;
see \cite{d2} for references.

An interval exchange transformation is defined as follows. Let $m >
1$ be a fixed integer and denote
$$
\Lambda_m = \{ \lambda \in \R^m : \lambda_j > 0 , \; 1 \le j \le m
\}
$$
and, for $\lambda \in \Lambda_m$,
\begin{align*}
\beta_j(\lambda) & = \begin{cases} 0 & j = 0 \\ \sum_{i=1}^j
\lambda_i & 1 \le j \le m \end{cases} \\
I^\lambda_j & = [\beta_{j-1}(\lambda),\beta_j(\lambda)) \\
|\lambda| & = \sum_{i = 1}^m \lambda_i \\
I^\lambda & = [0,|\lambda|).
\end{align*}
Denote by $\mathcal{S}_m$ the group of permutations on
$\{1,\ldots,m\}$, and set $\lambda_j^\pi = \lambda_{\pi^{-1}(j)}$
for $\lambda \in \Lambda_m$ and $\pi \in \mathcal{S}_m$. With these
definitions, the $(\lambda,\pi)$-interval exchange map
$T_{\lambda,\pi}$ is given by
$$
T_{\lambda,\pi} : I^\lambda \to I^\lambda , \; x \mapsto x -
\beta_{j-1}(\lambda) + \beta_{\pi(j) -1}(\lambda^\pi) \text{ for } x
\in I_j^\lambda , \; 1 \le j \le m.
$$
A permutation $\pi \in \mathcal{S}_m$ is called irreducible if
$\pi(\{ 1,\ldots ,k\}) = \{ 1,\ldots ,k\}$ implies $k = m$. We
denote the set of irreducible permutations by $\mathcal{S}_m^0$.

Let $\pi \in \mathcal{S}_m^0$. Then, for Lebesgue almost every
$\lambda \in \Lambda_m$, $(I^\lambda , T_{\lambda,\pi})$ is
strictly ergodic. The minimality statement follows from Keane's
work \cite{k}. The unique ergodicity statement was shown by Masur
\cite{m} and Veech \cite{v2}; for a simpler proof of this result,
see \cite{bosh3, bosh2, v3}. When unique ergodicity holds, it is
clear that the unique $T_{\lambda,\pi}$-invariant Borel
probability measure must be given by normalized Lebesgue measure
on $I^\lambda$, denoted by $\mathrm{Leb}$.

It is therefore natural to ask whether $(I^\lambda ,
T_{\lambda,\pi}, \mathrm{Leb})$ satisfies (MRP). We have the
following result:

\begin{theorem}\label{ietthm}
Let $\pi \in \mathcal{S}_m^0$. Then for Lebesgue almost every
$\lambda \in \Lambda_m$, $(I^\lambda , T_{\lambda,\pi})$ is
strictly ergodic and $(I^\lambda , T_{\lambda,\pi}, \mathrm{Leb})$
satisfies {\rm (MRP)}.
\end{theorem}

The proof of this theorem will rely on the following result due to
Veech; see \cite[Theorem~1.4]{v}.

\begin{theorem}[Veech 1984]\label{veechthm}
Let $\pi \in \mathcal{S}_m^0$. For Lebesgue almost every $\lambda
\in \Lambda_m$ and every $\varepsilon > 0$, there are $q \ge 1$
and an interval $J \subseteq I^\lambda$ such that
\begin{itemize}

\item[(i)] $J \cap T_{\lambda,\pi}^l J = \emptyset$, $1 \le l <
q$,

\item[(ii)] $T_{\lambda,\pi}$ is linear on $T_{\lambda,\pi}^l J$,
$0 \le l < q$,

\item[(iii)] $\mathrm{Leb}( \bigcup_{l = 0}^{q-1}
T_{\lambda,\pi}^l J ) > 1 - \varepsilon$,

\item[(iv)] $\mathrm{Leb}( J \cap T_{\lambda,\pi}^q J ) > (1 -
\varepsilon) \mathrm{Leb}(J)$.

\end{itemize}
\end{theorem}

\begin{proof}[Proof of Theorem~\ref{ietthm}.]
Consider a $\lambda$ from the full measure subset of $\Lambda_m$
such that $(I^\lambda , T_{\lambda,\pi})$ is strictly ergodic and
all the consequences listed in Theorem~\ref{veechthm} hold. We
claim that $(I^\lambda , T_{\lambda,\pi}, \mathrm{Leb})$ satisfies
{\rm (MRP)}.

For $\varepsilon > 0$ and $r > 0$, consider the set
$I^\lambda_{\varepsilon , r}$ of points $\omega \in I^\lambda$ for
which there exists $q \in \Z_+$ such that
$\mathrm{dist}(T_{\lambda,\pi}^k \omega , T_{\lambda,\pi}^{k+q}
\omega) < \varepsilon$ for $k = 0, 1, 2 , \ldots , \lfloor rq
\rfloor$. By Theorem~\ref{veechthm}, we have
$$
\mathrm{Leb} (I^\lambda \setminus I^\lambda_{\varepsilon , r})
\lesssim r \varepsilon.
$$
Thus, by Borel-Cantelli, the set
$$
\bigcap_{j,r \in \Z_+} I^\lambda_{2^{-j} , r}
$$
has full Lebesgue measure. Since this set is contained in
$PRP(I^\lambda , T_{\lambda,\pi})$, it follows that $(I^\lambda ,
T_{\lambda,\pi}, \mathrm{Leb})$ satisfies (MRP).
\end{proof}

\noindent\textit{Remarks.} (a) The astute reader may point out
that we have defined (MRP) only for homeomorphisms, and an
interval exchange transformation is in general discontinuous.
This, however, can be remedied in two ways. The first is to pass
to a symbolic setting, where we code an orbit by the sequence of
the exchanged intervals it hits. The standard shift transformation
on this sequence space over a finite alphabet (of cardinality $m$)
is then a homeomorphism and we can work in this representation;
compare, for example, \cite[Section~5]{k}. Notice that in the
strictly ergodic situation, the spectral consequences for the
associated Schr\"odinger operator family are the same because the
unique invariant measure in the symbolic setting is the
push-forward of Lebesgue measure. The other way to circumvent this
issue is to extend our result relating (MRP) to the almost sure
absence of eigenvalues for the associated Schr\"odinger operators
to certain discontinuous maps $T$.
\\[1mm]
(b) In light of the previous remark, it is interesting to point
out the following. The metric repetition property (MRP) is not
invariant under the passage to the symbolic setting. Indeed, while
we showed that every irrational rotation of the circle obeys
(MRP), its symbolic counterpart (a Sturmian sequence, resulting
from the symbolic coding of an exchange of two intervals)
satisfies (MRP) if and only if the rotation number has unbounded
partial quotients; see, for example, \cite{bd,dp,m2}.

\subsection{Uniformity in the Coupling
Constant}\label{couplingconstant}

One often introduces a coupling constant $\lambda$ and considers
potentials of the form
\begin{equation}\label{lambdapot}
V_\omega(n) = \lambda f(T^n \omega)
\end{equation}
instead of \eqref{pot}. Since the regimes of small and large
couplings can be regarded as small perturbations of simple models
(the free Laplacian and a diagonal matrix, respectively), it is of
especial interest to explore whether the spectral type of the limit
model extends to the perturbation.

In this context it should be noted that if $V_\omega$ of the form
\eqref{lambdapot} is a Gordon potential for one value of $\lambda$,
it is a Gordon potential for all values of $\lambda$. In particular,
the results on the absence of point spectrum above extend from
potentials of the form \eqref{pot} covered by our work to all
potentials of the more general form \eqref{lambdapot}.

\subsection{Generic Singular Continuous Spectrum}\label{scspec}

Our work is closely related in spirit to the paper \cite{ad} by
Avila and Damanik. It follows from \cite{ad} that there exists a
residual set $\mathcal{F}_{\mathrm{sing}} \subseteq C(\Omega)$ such
that for every $f \in \mathcal{F}_{\mathrm{sing}}$, the operator
\eqref{oper} with potential \eqref{pot} has purely singular spectrum
for almost every $\omega \in \Omega$, and moreover, for Lebesgue
almost every $\lambda \in \R$, the operator \eqref{oper} with
potential \eqref{lambdapot} has purely singular spectrum for almost
every $\omega \in \Omega$.

It is well known that discrete one-dimensional Schr\"odinger
operators with periodic potentials have purely absolutely continuous
spectrum. Thus, the results obtained in \cite{ad} and the present
paper are rather strong implementations of the philosophy that
absolutely continuous spectrum requires the presence of perfect
repetition, while a rather weak repetition property already ensures
the absence of eigenvalues.

Explicitly, the following consequence is obtained by combining the
results on the absence of eigenvalues and the results on the absence
of absolutely continuous spectrum.

\begin{theorem}\label{genscthm}
Suppose that $(\Omega,T,\mu)$ satisfies {\rm (MRP)}. Then, for $f$'s
from a residual subset of $C(\Omega)$, the operator \eqref{oper}
with potential \eqref{pot} has purely singular continuous spectrum
for $\mu$ almost every $\omega \in \Omega$, and moreover, the
operator \eqref{oper} with potential \eqref{lambdapot} has purely
singular continuous spectrum for $\mu$ almost every $\omega \in
\Omega$ and Lebesgue almost every $\lambda \in \R$.
\end{theorem}

Recall that by Theorems~\ref{shiftrpthm} and \ref{sshiftrpthm},
Theorem~\ref{genscthm} applies to all minimal shifts of the form
\eqref{qpmodels} and minimal skew-shifts of the form
\eqref{ssmodels}, where $\alpha$ is not badly approximable (which is
satisfied by Lebesgue almost every $\alpha$).

\subsection{A Remark on One-Frequency Quasi-Periodic Models}

Here we discuss a technical point that appears to be remarkable in
comparison with earlier studies of Gordon potentials generated by
shifts on $\T$ (a.k.a.\ rotations of the circle). There are three
truly distinct classes of sampling functions that lead to
technically very different theories on the level of Schr\"odinger
operators: piecewise constant functions with finitely many
discontinuities (a.k.a.\ codings of rotations), continuous
functions, and smooth functions. Very roughly speaking,
one-frequency quasi-periodic models with piecewise constant sampling
functions seem to have, as a rule, purely singular continuous
spectrum, whereas one-frequency quasi-periodic models with smooth
sampling functions seem to have purely absolutely continuous
spectrum for small coupling and pure point spectrum for large
coupling; see, for example, \cite{bg,bj,d2} and references therein.

Continuous sampling functions fall between these two classes and in
some sense, it is natural to use an approximation of a given
continuous function from either side (i.e., by piecewise continuous
functions or by smooth functions) in the study of such potentials.
While it had been expected that models with continuous sampling
function behave similarly to models with smooth sampling functions,
recent work (especially \cite{ad} and the present paper) has shown
that in fact they behave generically like models with piecewise
continuous step functions.

The technical point we would like to make is the following. While
\cite{ad} indeed used approximation by piecewise constant functions
as a key tool in the proof, in this paper we did not! We were able
to show a very general result: by Theorems~\ref{srpthm} and
\ref{shiftrpthm} we see that for any irrational shift on $\T$, a
generic continuous sampling will almost surely generate a Gordon
potential. It is clear from the proof that much stronger Gordon-type
conditions (with more repetitions and smaller error estimates) can
be obtained in the same way. This is surprising in the case of a
badly approximable $\alpha$. It can be shown, \cite{bd}, that for
any piecewise constant function with finitely many discontinuities
(that is not globally constant), such Gordon-type repetition
properties do not hold for badly approximable $\alpha$.
Consequently, the general result just described cannot be obtained
by approximation with piecewise constant functions. Of course, for
smooth functions, badly approximable $\alpha$ do not generate Gordon
potentials either. Thus, for shifts on $\T$ by a badly approximable
$\alpha$, the generic Gordon property in the continuous category is
a novel feature.

\subsection{Some Remarks on the Repetition
Properties}\label{ss.remarks}

It is a natural question whether there are any non-trivial
relations between the topological, metric, and global repetition
properties. For example, as was pointed out earlier, in all the
examples we have considered up to this point in the present paper,
the properties (TRP), (MRP), and (GRP) either hold or fail
simultaneously, so one could ask if there is a general principle.
The following result answers this question.

\begin{prop}
{\rm (a)} There are strictly ergodic examples of $(\Omega,T)$ that
satisfy {\rm (MRP)} but not {\rm (GRP)}.  \\
{\rm (b)} There are strictly ergodic examples of $(\Omega,T)$ that
satisfy {\rm (TRP)} but not {\rm (MRP)}.
\end{prop}

\begin{proof}
(a) The following general result holds: If $\mathcal{A}$ is a
finite set, $T$ is the standard shift-transformation on the
sequence space $\mathcal{A}^\Z$, and $\Omega$ is a closed,
$T$-invariant subset of $\mathcal{A}^\Z$ (a so-called subshift)
that contains points that are not $T$-periodic, then $(\Omega,T)$
does not satisfy (GRP). This can be derived from
\cite[Proposition~2.1]{bhz}.\footnote{There, only minimal
aperiodic subshifts are considered. If $(\Omega,T)$ is not
minimal, we can just choose a minimal component of the given
subshift and find an orbit without repetition property there.} In
particular, strictly ergodic examples can be constructed in this
way.

To describe an explicit example, take $\alpha$ with unbounded
partial quotients and consider the natural coding of the shift by
$\alpha$ on $\T$. The resulting Sturmian subshift satisfies (MRP)
by \cite{bd,dp} and it does not satisfy (GRP) by the general
result just mentioned.
\\
(b) Here we only describe a procedure that generates the desired
examples; we refer the reader to \cite{bd2}, where complete proofs
and a much more detailed discussion can be found. Consider
skew-products in the spirit of \eqref{ssmodels} on higher
dimensional tori. That is, let $\Omega = \T^d$ and let, for
example, $T : \Omega \to \Omega$ be given by
$$
T(\omega_1,\omega_2, \omega_3, \ldots, \omega_d) = (\omega_1 +
\alpha , \omega_1 + \omega_2, \omega_1 + \omega_2 + \omega_3 ,
\ldots , \omega_1 + \cdots + \omega_d),
$$
where $\alpha$ is irrational. Then, $(\Omega,T)$ is minimal and
normalized Lebesgue measure is the unique invariant probability
measure; see, for example, \cite{f}. If $d \ge 4$,
$(\Omega,T,\mathrm{Leb})$ does not satisfy (MRP) for any $\alpha$.
On the other hand, if $\alpha$ is sufficiently well approximated
by rational numbers, $(\Omega,T)$ does satisfy (TRP). Again, see
\cite{bd2} for proofs of these two statements in a more general
context.
\end{proof}

We conclude with a brief discussion of how the repetition properties
relate to notions of entropy. The following result establishes a
connection:

\begin{prop}\label{entropyprop}
Suppose $(\Omega,T)$ is a subshift over a finite set $\mathcal{A}$
and $\mu$ is an ergodic probability measure. If $(\Omega,T,\mu)$
satisfies {\rm (MRP)}, then it has metric entropy zero.
\end{prop}

\begin{proof}
This follows from \cite{ow}; see also \cite{bosh}.
\end{proof}

In view of Theorem~\ref{srpthm}, Proposition~\ref{entropyprop} is
nicely in line with the general expectation that positive entropy
strongly suggests positive Lyapunov exponents and localization for
the associated Schr\"odinger operators, at least in an
almost-everywhere sense. It would be interesting to see whether
there is a topological analogue. That is, is it true that the
validity of (TRP) implies zero topological entropy? Moreover, what
conclusions (in the spirit of entropy) can be drawn when a dynamical
system satisfies (GRP)?

\end{document}